\documentclass[preprint, 12pt, english]{amsart}
\usepackage{amsthm}
\usepackage{amsmath}
\usepackage{latexsym, amssymb}
\usepackage{txfonts}
\usepackage{mathtools}
\usepackage{color}
\usepackage[all]{xy}

\newtheorem{thm}{Theorem}[section] %the resolution could also be [subsection]

\newtheorem{claim}[thm]{Claim}

\newtheorem{cor}[thm]{Corollary}

\newtheorem{lem}[thm]{Lemma}
\newtheorem{prop}[thm]{Proposition}

\newtheorem{ques}[thm]{Question}

\theoremstyle{definition}
\newtheorem{rem}[thm]{Remark}

\newcommand\operA[2]{{\if!#2!\operatorname{#1}\else{\operatorname{#1}_{#2}^{\phantom{I}}}\fi}} % To be used within Bdefs. Usage: $\operA{N}{K/F}$ produces $N_{K/F}$; $\operA{N}{}$ produces $N$.

\renewcommand{\H}{\operatorname H}

\def\norm{{\operatorname{Norm}}}

\newcommand{\Trace}[1][]{\if!#1!\operatorname{Tr}\else{\operatorname{Tr}_{#1}^{\phantom{I}}}\fi} % Usage: $\Tr[K/F](a)$.

\long\def\forget#1\forgotten{{}} %

\def\({\left(}
\def\){\right)}

\newcommand\LAY[3][]{{\begin{array}{c}\mbox{#2} \if#1!{}\else{+}\fi \\ \mbox{#3}\end{array}}}

\makeatletter

\def\ps@pprintTitle{%
 \let\@oddhead\@empty
 \let\@evenhead\@empty
 \def\@oddfoot{}%
 \let\@evenfoot\@oddfoot}

\newcommand{\bigperp}{%
  \mathop{\mathpalette\bigp@rp\relax}%
  \displaylimits
}

\newcommand{\bigp@rp}[2]{%
  \vcenter{
    \m@th\hbox{\scalebox{\ifx#1\displaystyle2.1\else1.5\fi}{$#1\perp$}}
  }%
}
\makeatother

\renewcommand{\geq}{\geqslant}
\renewcommand{\leq}{\leqslant}

\newif\iffurther
\furtherfalse
\begin{document}
%\begin{frontmatter}

\title{Common Splitting Fields of Symbol Algebras}

\author{Adam Chapman}
\address{School of Computer Science, Academic College of Tel-Aviv-Yaffo, Rabenu Yeruham St., P.O.B 8401 Yaffo, 6818211, Israel}
\email{adam1chapman@yahoo.com}
\author{Mathieu Florence}
\address{Equipe de Topologie et G\'{e}om\'{e}trie Alg\'{e}briques, Institut de Math\'{e}matiques de Jussieu,
4 place Jussieu, 75005 Paris, France}
\email{mflorence@imj-prg.fr}
\author{Kelly McKinnie}
\address{Department of Mathematical Sciences, University of Montana, Missoula, MT 59812, USA}
\email{kelly.mckinnie@mso.umt.edu}

\begin{abstract}
We study the common splitting fields of symbol algebras of degree $p^m$ over fields $F$ of $\operatorname{char}(F)=p$.
We first show that if any finite number of such algebras share a degree $p^m$ simple purely inseparable splitting field, then they share a cyclic splitting field of the same degree.
As a consequence, we conclude that every finite number of symbol algebras of degrees $p^{m_0},\dots,p^{m_t}$ share a cyclic splitting field of degree $p^{m_0+\dots+m_t}$.
This generalization recovers the known fact that every tensor product of symbol algebras is a symbol algebra.
We apply a result of Tignol's to bound the symbol length of classes in $\operatorname{Br}_{p^m}(F)$ whose symbol length when embedded into $\operatorname{Br}_{p^{m+1}}(F)$ is 2 for $p\in \{2,3\}$.
We also study similar situations in other Kato-Milne cohomology groups, where the necessary norm conditions for splitting exist.
\end{abstract}

%\end{frontmatter}
\maketitle

\section{Introduction}

By classical results of Teichm\"{u}ller and Albert, the group $\operatorname{Br}_{p^m}(F)$ is generated by symbol algebras of degree $p^m$.
In \cite[Chapter 7]{Albert:1968}, Albert concludes that any central simple algebra of degree $p^m$ over a field $F$ of $\operatorname{char}(F)=p$ is Brauer equivalent to a symbol (cyclic) algebra of degree $p^t$ for some $t\geq m$. This is done by finding a common degree $p^{m+t}$ simple purely inseparable splitting field for any two symbol algebras of degrees $p^m$ and $p^t$, which implies that their tensor product is a symbol algebra of degree $p^{m+t}$.
We generalize this result by finding a common cyclic splitting field of degree $p^{m_0+\dots+m_t}$ for any finite number of such algebras of degrees $p^{m_0},\dots,p^{m_t}$. This is a result of the Theorem \ref{cyclic}, which states that any finite number of symbol algebras of the same degree sharing a simple purely inseparable maximal subfield also share a cyclic maximal subfield.

Tignol studied the opposite direction in \cite{Tignol:1983}, proving that a symbol algebra of degree $p^t$ and exponent $p^m$ is of symbol length at most $p^{t-m}$ in $\operatorname{Br}_{p^m}(F)$.
We apply this theorem in bounding the symbol length in $\operatorname{Br}_{p^m}(F)$ of classes that are Brauer equivalent to tensor products of two cyclic algebras of degree $p^{m+1}$  when $p$ is either 2 or 3, based on the known chain lemmas for cyclic algebras of degree 2 and 3, using a method introduced earlier in \cite{Matzri:2014}.
We then consider what happens for more general Kato-Milne cohomology groups in cases where the necessary norm conditions exist.

\section{Preliminaries}

Though our primary motivation and interest are in algebras and the Brauer group, the Kato-Milne cohomology groups give us the proper setting to prove our results as well as a way to generalize beyond the setting of algebras.
For a field $F$ of characteristic $p$, positive integer $m$ and nonnegative integer $n$, the Kato-Milne coholology group $\H_{p^m}^{n+1}(F)$ is the additive group $W_m(F) \otimes \underbrace{F^\times \otimes \dots \otimes F^\times}_{n \ \text{times}}$ modulo the relations
\begin{itemize}
\item $(\omega^p-\omega) \otimes b_1 \otimes \dots \otimes b_n=0$, 
\item $(0\dots0,a,0,\ldots,0) \otimes a \otimes b_2 \otimes \dots \otimes b_n=0$, and 
\item $\omega \otimes b_1 \otimes \dots \otimes b_n=0$ where $b_i=b_j$ for some $i \neq j$.
\end{itemize}
Here $W_m(F)$ is the ring of truncated Witt vectors of length $m$ over $F$, and for each $\omega=(\omega_1,\dots,\omega_m)$, $\omega^p$ stands for $(\omega_1^p,\dots,\omega_m^p)$. For a comprehensive reference on these groups see \cite{AravireJacobORyan:2018} or the more classical reference \cite{Izhboldin:2000}.
The generators $\omega \otimes b_1 \otimes \dots \otimes b_n$ are called ``($p^m$-)symbols". For $n=1$, these groups describe the $p^m$-torsion of the Brauer group, i.e., $\operatorname{H}_{p^m}^2(F) \cong \operatorname{Br}_{p^m}(F)$ with the isomorphism given by $\omega \otimes b \mapsto [\omega,b)_F$, where $[\omega,b)_F$ stands for the symbol algebra generated by $\theta_1,\dots,\theta_m$ and $y$ satisfying
$$\vec{\theta}^p-\vec{\theta}=\omega, \qquad y^{p^m}=b, \qquad \text{and} \qquad y\,\vec{\theta}\,y^{-1}=\vec{\theta}+\vec{1}$$
where $\vec{\theta}=(\theta_1,\theta_2,\dots,\theta_m)$ is a truncated Witt vector, $\vec{\theta}^p=(\theta_1^p,\theta_2^p,\dots,\theta_m^p)$, and $\vec{1}=(1,0,\dots,0)$ (see \cite{MammoneMerkurjev:1991} for reference). 
For simplicity, we denote by $F_\omega$ the degree $p^m$ cyclic subfield of $[\omega,\beta)_F$ generated by $\theta_1,\dots,\theta_m$.
The symbol length of a class in $\H_{p^m}^{n+1}(F)$ is the minimal $t$ for which the class can be written as the sum of $t$ symbols. 
For any $\ell\in \{1,\dots,m-1\}$, there is a shift map $\operatorname{Shift}_{m-\ell}^\ell$ from the group $\operatorname{H}_{p^{m-\ell}}^{n+1}(F)$ to $\operatorname{H}_{p^m}^{n+1}(F)$ given by
$(a_1,\dots,a_{m-\ell}) \otimes b_1 \otimes \dots \otimes b_n \mapsto (\underbrace{0,\dots,0}_{\ell \ \text{times}},a_1,\dots,a_{m-\ell}) \otimes b_1 \otimes \dots \otimes b_n$.
In the opposite direction, there is the map taking each $\pi \in \operatorname{H}_{p^m}^{n+1}(F)$ to $\underbrace{\pi + \dots + \pi}_{p^{\ell} \ \text{times}}$. In particular, it takes each symbol $(a_1,\dots,a_m) \otimes b_1 \otimes \dots \otimes b_n$ to $(0,\dots,0,a_1^{p^{\ell}},\dots,a_{m-\ell}^{p^{\ell}}) \otimes b_1 \otimes \dots \otimes b_n$ (which is equal to $(0,\dots,0,a_1,\dots,a_{m-\ell}) \otimes b_1 \otimes \dots \otimes b_n$ in this group). It gives rise to a homomorphism $\operatorname{Exp}_m^\ell : \H_{p^m}^{n+1}(F) \rightarrow \H_{p^{m-\ell}}^{n+1}(F)$ which maps every symbol $(a_1,\dots,a_m) \otimes b_1 \otimes \dots \otimes b_n$ to $(a_1,\dots,a_{m-\ell}) \otimes b_1 \otimes \dots \otimes b_n$.
It is clear from the formulas that $\operatorname{Exp}_m^\ell\circ \operatorname{Shift}^{m-t}_t=\operatorname{Exp}_t^\ell$ for any $t<m$.
\begin{thm}[{\cite[Theorem 2.31]{AravireJacobORyan:2018}; see also \cite[Theorem 1]{Izhboldin:2000} and \cite[Lemma 6.2]{Izhboldin:1996}}]\label{Exact}
The following sequence is exact
$$\xymatrix{
0\ar@{->}[r] & \H_{p^{m-1}}^{n+1}(F)\ar@{->}^{\operatorname{Shift}_{m-1}^1}[r] &  \H_{p^m}^{n+1}(F)\ar@{->}[r]^{\operatorname{Exp}_m^{m-1}} &  \H_{p}^{n+1}(F)\ar@{->}[r] & 0.}$$
\end{thm}

\begin{cor}[{\cite[Remark 2.32]{AravireJacobORyan:2018}}]
The following sequence is also exact
$$\xymatrix{
0\ar@{->}[r] & \H_{p^{m-\ell}}^{n+1}(F)\ar@{->}^{\operatorname{Shift}_{m-\ell}^\ell}[r] &  \H_{p^m}^{n+1}(F)\ar@{->}[r]^{\operatorname{Exp}_m^{m-\ell}} &  \H_{p^\ell}^{n+1}(F)\ar@{->}[r] & 0.}$$
\end{cor}

%\begin{proof}
%The case of $\ell=1$ is covered in Theorem \ref{Exact}.
%Suppose the statement holds true for all positive integers strictly smaller than $\ell$, then consider a class $A$ in $\H_{p^m}^{n+1}(F)$ with $\operatorname{Exp}_m^{m-\ell}(A)=0$.
%Then $\operatorname{Exp}_m^{m-\ell+1}(A)=0$ as well, which means that there exists $B \in \H_{p^{m-\ell+1}}^{n+1}(F)$ for which $\operatorname{Shift}_{m-\ell+1}^{\ell-1}(B)=A$.
%Write $t=m-\ell+1$. Since $\operatorname{Exp}_t^{t-1}(B)=\operatorname{Exp}_m^{m-\ell}(A)=0$, by Theorem \ref{Exact}, there exists $C \in \H_{p^{t-1}}^{n+1}(F)=\H_{p^{m-\ell}}^{n+1}(F)$ for which $\operatorname{Shift}_{t-1}^1(C)=B$, and so $\operatorname{Shift}_{m-\ell}^{\ell}(C)=\operatorname{Shift}_{m-\ell+1}^{\ell-1}\circ \operatorname{Shift}_{t-1}^{1}(C)=A$.
%\end{proof}
For the sake of simplicity, from here on we consider $\H_{p^{m-\ell}}^{n+1}(F)$ the subgroup of $\H_{p^m}^{n+1}(F)$ consisting of the elements whose exponent divides $p^{m-\ell}$, identifying this subgroup with $\operatorname{Shift}_{m-\ell}^{\ell}(\H_{p^{m-\ell}}^{n+1}(F))$. The group $\H_{p^\infty}^{n+1}(F)$ is thus defined to be $\lim_{m\rightarrow \infty} \H_{p^m}^{n+1}(F)$.

Recall also the connection (see \cite{Kato:1982} and \cite{EKM} for background) between $\H_2^{n+1}(F)$ and quadratic forms:
\begin{eqnarray*}
 \H_2^{n+1}(F) & \cong & I_q^{n+1} F/I_q^{n+2} F\\
\alpha \otimes \beta_1 \otimes \dots \otimes \beta_n &\mapsto & \langle \! \langle \beta_1,\dots,\beta_n,\alpha ]\!].
\end{eqnarray*}
This will be of significance when we get to the norm conditions for $p=2$.
\section{Norm Conditions}\label{section:norm}

Given a symbol $A \in \H^n_{p^m}(F)$ and $c \in F^\times$, in certain cases the triviality of $A \otimes c$ in $\H_{p^m}^{n+1}(F)$ implies that a norm condition is satisfied.
We list the known cases in the following theorems.
These will be used in the proofs in Section 4.

\begin{thm}[{\cite[Corollary 4.7.5]{GilleSzamuely:2017}}]\label{Brauer}
If a symbol $\omega \in \H_{p^m}^1(F)$ and $c\in F^\times$ satisfy $\omega\otimes c=0$ in $\H_{p^m}^2(F)$ then $c$ is a norm in $F_\omega$. 
\end{thm}

\begin{thm}[{\cite[Th\'eor\`eme 6]{Gille:2000}}]\label{GilleTheorem}
If a symbol $A \in \H_p^2(F)$ and $c\in F^\times$ satisfy $A \otimes c=0$ in $\H_p^3(F)$, then $c$ is a reduced norm in the division algebra representative of the class $A$.
\end{thm}

\begin{thm}\label{QuadraticNorm}
If a symbol $A \in \H_2^n(F)$ and $c\in F^\times$ satisfy $A \otimes c=0$ in $\H_2^{n+1}(F)$, then $c$ is represented by the quadratic $n$-fold Pfister form representative of the class $A$.
\end{thm}

\begin{proof}
The class of $A$ is represented by a quadratic $n$-fold Pfister form $\varphi=\langle \! \langle \beta_1,\dots,\beta_{n-1},\alpha ] \!]$.
Since $A \otimes c=0$ in $\H_2^{n+1}(F)$, the $(n+1)$-fold Pfister form $\langle \! \langle c \rangle \! \rangle \otimes \varphi$ is hyperbolic, and therefore its Pfister neighbor $\langle c \rangle \perp \varphi$ is isotropic, which in turn implies that $c$ is represented by $\varphi$ by \cite[Proposition 9.8]{EKM}.
%If $\varphi$ is isotropic then it is hyperbolic, and therefore universal and the statement is trivial.
%Suppose that $\varphi$ is anisotropic.
%Since $A \otimes c=0$ in $\H_2^{n+1}(F)$, the $(n+1)$-fold Pfister form $\langle \! \langle c \rangle \! \rangle \otimes \varphi$ is hyperbolic, and therefore its Pfister neighbor $\langle c \rangle \perp \varphi$ is isotropic.
%Since $\varphi$ is anisotropic, this means $c$ is represented by $\varphi$ by \cite[Section 23]{EKM}.
\end{proof}

\section{Common Splitting Fields}

\begin{lem}[{\cite[9.1.11]{GilleSzamuely:2017}}] \label{GilleLemma}
Let $p$ be a prime integer, $F$ a field of $\operatorname{char}(F)=p$ and $K=F(\sqrt[p^r]b)$ a purely inseparable extension of degree $p^r$, and let $L/F$ be a finite separable extension. Then there exists an element $v \in LK$ whose norm, $\norm_{LK/K}(v)$, generates the extension $K/F$.
\end{lem}

\begin{prop}[{cf. \cite[9.1.9]{GilleSzamuely:2017}}]\label{inseparablecommon} Let $p$ be a prime integer, $F$ a field of $\operatorname{char}(F)=p$, and for $1\leq i \leq m$ let $A_i$ be a symbol $F$-algebra of degree $p^{r_i}$. Then there exists an element $b \in F$ such that the extension $F(\sqrt[p^r]{b})$ splits all the $A_i$'s where $r\leq \sum_{i=1}^mr_i$.
\end{prop}

\begin{proof}
The proof is by induction on $m$. The case $m=2$ is Albert’s (see \cite[9.1.9]{GilleSzamuely:2017}) and the induction step argument below follows Albert’s proof.
Suppose it holds true for any integer smaller than $m$.
Then there exists a common splitting field $E=F[\sqrt[p^t]{b}]$ for $A_1,\dots,A_{m-1}$ where $t\leq \sum_{i=1}^{m-1} r_i$.
Write $A_m=[\omega,\gamma)$.
Then by Lemma \ref{GilleLemma}, by taking $K=E$ and $L=F_\omega$ (which is a cyclic subfield of $A_m$), there exists an element $v \in EL$  whose norm $z=\norm_{LE/E}(v)$ generates $E/F$.
Now, take $T=E(\sqrt[p^{r_m}]{z\gamma})$.
Clearly $T$ splits $A_1,\dots,A_{m-1}$.
Since $z$ generates $E$ over $F$, the element
$y=\sqrt[p^{r_m}]{z\gamma}$ generates $T$ over $F$.
It remains to explain why $T$ splits $A_m$.
$$A_m \otimes T=[\omega,\gamma) \otimes T=[\omega,y^{p^{r_m}}z^{-1}) \otimes T=[\omega,y^{p^{r_m}}) \otimes [\omega,z^{-1}) \otimes T.$$
The algebra $[\omega,y^{p^{r_m}})$ is split because $[\omega,y^{p^{r_m}})=[\underbrace{\omega+\dots+\omega}_{p^{r_m} \ \text{times}},y)$ and $\omega \in H_{p^{r_m}}^1(F)$.
The algebra $[\omega,z^{-1})\otimes T$ is split because $z$ is a norm in $EL/E$ (and so also in $TL/T$).
\end{proof}

\begin{cor}\label{sl1}
Every class $A$ in $\H_{p^m}^{n+1}(F)$ is a single symbol in $\H_{p^t}^{n+1}(F)$ for a large enough $t$. More precisely, $t \leq m\ell^n$ where $\ell$ is the symbol length of $A$.
\end{cor}

\begin{proof}
Take a class $A$ in $\H_{p^m}^{n+1}(F)$ and write it as a sum of symbols 
$$A=\sum_{i=1}^\ell \omega_i \otimes \beta_{1,i} \otimes \dots \otimes \beta_{n,i}.$$
Since each $\omega_i \otimes \beta_{1,i}$ is the class of the symbol algebras $[\omega_i,\beta_{1,i})_{p^m,F}$, 
by Proposition \ref{inseparablecommon} there exists a common purely inseparable splitting field $F[\sqrt[p^r]{c}]$ where $r\leq m\ell$, and thus there exist $\omega_1',\dots,\omega_\ell' \in W_r(F)$ such that for each $i\in \{1,\dots,\ell\}$, $[\omega_i,\beta_{1,i})_{p^m,F}=[\omega_i',c)_{p^r,F}$ (see \cite[Chapter VII, Theorem 28]{Albert:1968} or \cite[Theorem 9.1.1]{GilleSzamuely:2017}).
Therefore, as an element of $\H_{p^r}^{n+1}(F)$, $A$ can be written as
$$A=\sum_{i=1}^\ell \omega_i' \otimes c \otimes \beta_{2,i} \otimes \dots \otimes \beta_{n,i}.$$
Since for each $i$, $\omega_i' \otimes c \otimes \beta_{2,i} \otimes \dots \otimes \beta_{n,i}=(-1)^{n-1}\omega_i'  \otimes \beta_{2,i} \otimes \dots \otimes \beta_{n,i} \otimes c$, this process can be repeated until we obtain an integer $t$ such that as an element of $\H_{p^t}^{n+1}(F)$, $A$ can be written as
$$A=\sum_{i=1}^\ell \rho_i \otimes c \otimes c_2 \otimes \dots \otimes c_n,$$
for some $\rho_1,\dots,\rho_\ell \in W_t(F)$ and $c_2,\dots,c_n \in F^\times$. But then $A=(\rho_1+\dots+\rho_\ell) \otimes c \otimes c_2 \otimes \dots \otimes c_n$ is a single symbol in $\H_{p^t}^{n+1}(F)$.
\end{proof}

We now turn to proving that sharing purely inseparable simple splitting fields implies sharing cyclic splitting fields.
This requires some preparation:

\begin{prop}\label{depend}
Let $p$ be a prime integer, $t$ a positive integer, $F$ a field of $\operatorname{char}(F)=p$ and $r_0,\dots,r_t$ be $\mathbb{F}_p$-independent elements in $F$.
Then $\wp(\frac{r_1}{r_0}),\dots,\wp(\frac{r_t}{r_0})$ are $\mathbb{F}_p$-independent as well.
\end{prop}

\begin{proof}
Suppose there exist $c_1,\dots,c_t \in \mathbb{F}_p$, not all zero, such that
$$c_1\wp(\frac{r_1}{r_0})+\dots+c_t\wp(\frac{r_t}{r_0})=0.$$
Since for each $i \in \{1,\dots,t\}$, $c_i=c_i^p$, we obtain
$$\wp(c_1\frac{r_1}{r_0})+\dots+\wp(c_t\frac{r_t}{r_0})=0,$$
and so
$$\wp(c_1\frac{r_1}{r_0}+\dots+c_t\frac{r_t}{r_0})=0.$$
This means $c_1\frac{r_1}{r_0}+\dots+c_t\frac{r_t}{r_0}$ is a root of the polynomial $\wp(\lambda)=\lambda^p-\lambda$, but these roots are exactly the elements of $\mathbb{F}_p$.
Therefore, $c_1\frac{r_1}{r_0}+\dots+c_t\frac{r_t}{r_0}=c_0$ for some $c_0 \in \mathbb{F}_p$, and so
$$c_1 r_1+\dots+c_t r_t-c_0 r_0=0,$$
contradiction.
\end{proof}

\begin{lem}\label{WittVec}
Let $t$ be a nonnegative integer, $m$ a positive integer and $F$ be a field of $\operatorname{char}(F)=p>0$, and let $r_0,\dots,r_t \in F$ be $\mathbb{F}_p$-independent.
Then for any $\psi_0,\dots,\psi_t \in F$ there exists $\pi \in F$ for which $\psi_i \equiv r_i^{p^t} \pi \pmod{\wp(F)}$ for any $i \in \{0,\dots,t\}$.
\end{lem}

\begin{proof}
By induction on $t$.
The statement is clearly correct for $t=0$.
Suppose it is correct for $t-1$.
Consider the system
\begin{eqnarray*}
\psi_0 & \equiv &  r_0^{p^t} \pi \pmod{\wp(F)}\\
\vdots & & \\
\psi_t & \equiv &  r_t^{p^t} \pi \pmod{\wp(F)}\\
\end{eqnarray*}
A solution will be obtained if we can find $X_0,\dots,X_t \in F$ for which 
\begin{eqnarray*}
\psi_0 & = &  r_0^{p^t} \pi+X_0^p-X_0\\
\vdots & & \\
\psi_t & = &  r_t^{p^t} \pi+X_t^p-X_t\\
\end{eqnarray*}
which is equivalent to 
\begin{eqnarray*}
\frac{1}{r_0^{p^t}}\psi_0 & = &  \pi+\frac{1}{r_0^{p^t}}(X_0^p-X_0)\\
\vdots & & \\
\frac{1}{r_t^{p^t}}\psi_t & = &  \pi+\frac{1}{r_t^{p^t}}(X_t^p-X_t)\\
\end{eqnarray*}

By subtracting the first equation from each of the other equations and plugging in $X_i=(T_i+\frac{r_i^{p^{t-1}}}{r_0^{p^{t-1}}}X_0)$ for $i \geq 1$, we get
\begin{eqnarray*}
\frac{1}{r_1^{p^t}}\psi_1-\frac{1}{r_0^{p^t}} \psi_0 & = &  \frac{1}{r_1^{p^t}}(T_1^p-T_1)-(\frac{1}{r_1^{p^{t-1}(p-1)}r_0^{p^{t-1}}}-\frac{1}{r_0^{p^t}})X_0\\
\vdots & & \\
\frac{1}{r_t^{p^t}}\psi_t-\frac{1}{r_0^{p^t}} \psi_0 & = &  \frac{1}{r_t^{p^t}}(T_t^p-T_t)-(\frac{1}{r_t^{p^{t-1}(p-1)}r_0^{p^{t-1}}}-\frac{1}{r_0^{p^t}})X_0\\
\end{eqnarray*}
Which is equivalent to
\begin{eqnarray*}
\psi_1-\frac{r_1^{p^t}}{r_0^{p^t}} \psi_0 & \equiv &  (\frac{r_1^{p^{t-1}}}{r_0^{p^{t-1}}}-\frac{r_1^{p^t}}{r_0^{p^t}})X_0  \pmod{\wp(F)}\\
\vdots & & \\
\psi_t-\frac{r_t^{p^t}}{r_0^{p^t}} \psi_0 & \equiv &  (\frac{r_t^{p^{t-1}}}{r_0^{p^{t-1}}}-\frac{r_t^{p^t}}{r_0^{p^t}})X_0  \pmod{\wp(F)}\\
\end{eqnarray*}
The latter satisfies the induction hypothesis because
for each $i \in \{1,\dots,t\}$, $\frac{r_i^{p^{t-1}}}{r_0^{p^{t-1}}}-\frac{r_i^{p^t}}{r_0^{p^t}}=(\wp(-\frac{r_i}{r_0}))^{p^{t-1}}$ and the elements $\wp(\frac{r_1}{r_0}),\dots,\wp(\frac{r_t}{r_0})$ are $\mathbb{F}_p$-independent by Proposition \ref{depend}, and therefore has a solution. Hence, the original system has a solution.
\end{proof}

\begin{rem}
It is an immediate result of Lemma \ref{WittVec} that the essential dimension of $(\mathbb{Z}/p\mathbb{Z})^{\times (t+1)}$ over fields $F$ of $\operatorname{char}(F)=p$ with $|F| \geq p^{t+1}$ is 1, a fact that is known in the literature, see \cite[Remark 3.8]{BerhuyFavi:2003} and \cite[Lemma 2]{Ledet:2004}).% It is interesting to note that here we merely require that $|F| \geq p^{t+1}$, whereas in \cite{BerhuyFavi:2003} they seem to require that $F$ contains $\mathbb{F}_{p^{t+1}}$ specifically, which is a stronger assumption.
\end{rem}

\begin{thm}\label{cyclic}
Let $t$ be a positive integer and $A_0,A_1,\dots,A_t$ be $t+1$ symbol algebras of degree $p^m$ over an infinite field $F$ of $\operatorname{char}(F)=p>0$ and let $F[\sqrt[p^m]{\beta}]$ be a common splitting field of theirs. Take $r_1,\dots,r_t \in F$ which are $\mathbb{F}_p$-independent. Then there exists $\omega \in W_m(F)$ such that $A_i=[\omega,\beta-r_i^{p^{m+t-1}})_{p^m,F}$ for each $i \in \{1,\dots,t\}$ and $A_0=[\omega,\beta)_{p^m,F}$.
\end{thm}

\begin{proof}
Since $F[\sqrt[p^m]{\beta}]=F[\sqrt[p^m]{\beta-r_i^{p^{m+t-1}}}]$ is a splitting field of all the $A_i$'s, one can write $A_i=[\omega_i,\beta-r_i^{p^{m+t-1}})_{p^m,F}$ for each $i \in \{1,\dots,t\}$ and $A_0=[\omega_0,\beta)_{p^m,F}$ for some $\omega_0,\dots,\omega_t \in W_m(F)$.

Write $\rho_i$ for the first slot of the Witt vector $\omega_i$, and $\psi_i=\rho_i-\rho_0$ for each $i \in \{1,\dots,t\}$.
Then by Lemma \ref{WittVec}, there exists $\pi \in F$ for which $\psi_i \equiv r_i^{p^{t-1}} \pi \pmod{\wp(F)}$ for each $i$.
Then $\rho_0 \equiv \rho_i-r_i^{p^{t-1}} \pi \pmod{\wp(F)}$,
and by raising both sides to the power of $p^m$ we get $\rho_0^{p^m} \equiv \rho_i^{p^m}-r_i^{p^{m+t-1}} \pi^{p^m} \pmod{\wp(F)}$.
Therefore, $\rho_0^{p^m}+\beta \pi^{p^m} \equiv \rho_i^{p^m}+(\beta-r_i^{p^{m+t-1}}) \pi^{p^m} \pmod{\wp(F)}$.

Then by the symbol modifications $A_0=[\omega_0,\beta)_{p^m,F}=[\omega_0^{p^m},\beta)_{p^m,F}=[\omega_0^{p^m}+(\beta \pi^{p^m},0,\dots,0),\beta)_{p^m,F}$, and for each $i \in \{1,\dots,t\}$,
\begin{eqnarray*}
A_i&=&[\omega_i,\beta-r_i^{p^{m+t-1}})_{p^m,F}\\
&=&[\omega_i^{p^m},\beta-r_i^{p^{m+t-1}})_{p^m,F}\\
&=&[\omega_i^{p^m}+((\beta-r_i^{p^{m+t-1}})\pi^{p^m},0,\dots,0),\beta-r_i^{p^{m+t-1}})_{p^m,F}
\end{eqnarray*}
 we obtain symbol presentations where all the Witt vectors share the first slot.
This process can be repeated to modify the Witt vectors again so that they share the second slot too, while still sharing the first slot, and so on.
In the end, we obtain symbol presentations where the Witt vectors are identical.
\end{proof}

\begin{cor} \label{SuperAlbert}
For every finite number of symbol algebras $A_0,A_1,\dots,A_t$ of degrees $p^{m_0},\dots,p^{m_t}$ over a field $F$ of $\operatorname{char}(F)=p>0$, there exists a common cyclic splitting field of degree $p^{m_0+\dots+m_t}$ over $F$. 
\end{cor}

\begin{proof}
By Proposition \ref{inseparablecommon}, these algebras share a simple purely inseparable splitting field of degree $p^{m_0+\dots+m_t}$. Thus, these algebras are Brauer equivalent to symbol algebras of degree $p^{m_0+\dots+m_t}$ sharing such a splitting field. By the previous theorem, these algebras share a cyclic splitting field of degree $p^{m_0+\dots+m_t}$.\end{proof}

In particular, this recovers the known fact that $A_0 \otimes \dots \otimes A_t$ is a cyclic algebra of degree $p^{m_0+\dots+m_t}$ (\cite[Chapter VII, Theorem 31]{Albert:1968}). In terms of cohomology, this means that 
the Brauer class of this tensor product is a single symbol in $\H_{p^m}^2(F)$ where $m=m_0+\dots+m_t$.

\begin{rem}
Theorem \ref{cyclic} generalizes the main theorem of \cite{Chapman:2015} which states that for any $\alpha,\beta \in F$ and $\gamma \in F^\times$, the algebras $[\alpha,\gamma)_{p,F}$ and $[\beta,\gamma)_{p,F}$ share a degree $p$ cyclic splitting field.
As demonstrated in that paper, the converse is not true in general. Indeed, using the same argument, one can easily show that if $F=\mathbb{F}_p(\!(\alpha)\!)(\!(\beta)\!)$, then $A=[(1,0\dots,0),\alpha)_{p^m,F}$ and $B=[(1,0,\dots,0),\beta)_{p^m,F}$ do not share a common simple purely inseparable splitting field of degree $p^m$.
\end{rem}

\begin{rem}
Theorem \ref{cyclic} does not extend to infinite sets of symbol algebras.
For example, if one considers $F=\mathbb{F}_p(\alpha)$, then there is no single cyclic degree $p$ extension that trivializes the entire group $\H_{p}^{2}(F)$, even though this group is trivialized by the inseparable extension $F[\sqrt[p]{\alpha}]$.
\end{rem}
\section{Symbol Length moving from one group to another}

As a result of Corollary \ref{sl1} , the symbol length in $\H_{p^\infty}^{n+1}(F)$ is 1. It is only natural to ask what happens in the opposite direction:
\begin{ques}
If $A$ is a class in $\H_{p^m}^{n+1}(F)$ which is represented by a single symbol when embedded into $\H_{p^t}^{n+1}(F)$ for $t \geq m$, what is the symbol length of $A$ in $\H_{p^m}^{n+1}(F)$?
\end{ques}

The following theorems provide upper bounds for the symbol length of $A$ in $\H_{p^m}^{n+1}(F)$ under the conditions described in Section \ref{section:norm}.

\begin{thm}[{\cite{Tignol:1983}}]\label{BrauerSymbol}
Suppose $F$ is a field of $\operatorname{char}(F)=p$.
Then for any symbol algebra $A=[\omega,\gamma) \in \H^2_{p^t}(F)$ of exponent dividing  $p^m$ with $m\leq t$, its symbol length in $\H_{p^m}^2(F)$ is at most $p^{t-m}$.
\end{thm}

\begin{proof}
The original theorem was stated in the case of $\operatorname{char}(F) \neq p$ when $F$ contains primitive $p^t$th roots of unity.
However, the original proof translates smoothly to our case.
The required result on the corestriction used in the original proof was simply published a few years later in \cite{MammoneMerkurjev:1991}.
\end{proof}

\begin{cor}\label{TignolBound}
Suppose $F$ is a field of $\operatorname{char}(F)=p$.
Then for any central simple $F$-algebra whose exponent is $p^m$ and whose symbol length in $\H_{p^\ell}^2(F)$ for some $\ell \geq m$ is $r$, its symbol length in $\H_{p^m}^2(F)$ is at most $p^{r\ell-m}$.
\end{cor}

\begin{proof}
Follows from the previous theorem  and  by the fact that a tensor product of $r$ symbols in $\H_{p^\ell}^2(F)$ is a single symbol in $\H_{p^{r\ell}}^2(F)$ (Corollary \ref{SuperAlbert}).
\end{proof}

\begin{lem}\label{commonslotchar2}
\
\begin{enumerate}
\item Given Witt vectors $\alpha=(\alpha_1,\dots,\alpha_{m+1})$ and $\gamma=(\gamma_1,\dots,\gamma_{m+1})$ in $W_{m+1}(F)$, and $\beta \in F^\times$, if $[\alpha_1,\beta)_{p,F}=[\gamma_1,\beta)_{p,F}$, then $[\alpha,\beta)_{p^{m+1},F}^{\operatorname{op}} \otimes [\gamma,\beta)_{p^{m+1},F}$ is Brauer equivalent to a single symbol in $\H_{p^m}^2(F)$.
\item Given a Witt vector $\alpha=(\alpha_1,\dots,\alpha_{m+1})$ and $\beta,\delta \in F^\times$, if $[\alpha_1,\beta)_{p,F}=[\alpha_1,\delta)_{p,F}$, then $[\alpha,\beta)_{p^{m+1},F}^{\operatorname{op}} \otimes [\alpha,\delta)_{p^{m+1},F}$ is of symbol length at most $p$ in $\H_{p^m}^2(F)$.
\end{enumerate}
\end{lem}

\begin{proof}
For the first statement, if $[\alpha_1,\beta)_{p,F}=[\gamma_1,\beta)_{p,F}$ then $\gamma_1=\alpha_1+\lambda^p-\lambda+c_1^p \beta+\dots+c_{p-1}^p \beta^{p-1}$ for some $\lambda,c_1,\dots,c_{p-1} \in F$. Therefore, $[\gamma,\beta)_{p^{m+1},F} \otimes [-\alpha,\beta)_{p^{m+1},F}=[\gamma-\alpha,\beta)_{p^{m+1},F}=
[\gamma-\alpha-(\lambda^p,0,\dots,0)+(\lambda,0,\dots,0)-(c_1^p \beta,0,\dots,0)-\dots-(c_{p-1}^p \beta^{p-1},0,\dots,0),\beta)_{p^{m+1},F}$.
Since the first slot of $\gamma-\alpha-(\lambda^p,0,\dots,0)+(\lambda,0,\dots,0)-(c_1^p \beta,0,\dots,0)-\dots-(c_{p-1}^p \beta^{p-1},0,\dots,0)$ is 0, this Witt vector is $\operatorname{Shift}_m^1(\pi)$ for some $\pi \in W_m(F)$.
Therefore,  $[\gamma,\beta)_{p^{m+1},F} \otimes [-\alpha,\beta)_{p^{m+1},F} \sim_{\operatorname{Br}} [\pi,\beta)_{p^m,F}$.

For the second statement, if $[\alpha_1,\beta)_{p,F}=[\alpha_1,\delta)_{p,F}$, then $\delta=\beta \cdot \norm_{K/F}(f)$ for some $f\in K=F_{\alpha_1}$. Therefore, $[\alpha,\delta)_{p^{m+1},F}=[\alpha,\beta \cdot \norm_{K/F}(f))_{p^{m+1},F} \sim_{\operatorname{Br}} [\alpha,\norm_{K/F}(f))_{p^{m+1},F} \otimes [\alpha,\beta)_{p^{m+1},F}$.
Since $[\alpha,\norm_{K/F}(f))_{p^{m+1},F}$ is of exponent dividing $p^m$, its symbol length in $\operatorname{Br}_{p^m}(F)$ is bounded from above by $p$ by Theorem \ref{BrauerSymbol}.
\end{proof}

\begin{cor}\label{Small}
\
\begin{itemize}
\item[a.] If $\operatorname{char}(F)=2$ then any class in $\operatorname{Br}_{2^m}(F)$ of symbol length at most 2 in $\operatorname{Br}_{2^{m+1}}(F)$ has symbol length at most 4 in $\operatorname{Br}_{2^m}(F)$.
\item[b.] If $\operatorname{char}(F)=3$ then any class in $\operatorname{Br}_{3^m}(F)$ of symbol length at most 2 in $\operatorname{Br}_{3^{m+1}}(F)$ has symbol length at most $11$ in $\operatorname{Br}_{3^m}(F)$
\end{itemize}
\end{cor}

\begin{proof}
Let $\alpha=(\alpha_1,\dots,\alpha_{m+1})$ and $\gamma=(\gamma_1,\dots,\gamma_{m+1})$ be two Witt vectors in $W_{m+1}(F)$.
Suppose $\operatorname{char}(F)=2$.
If $[\alpha,\beta)_{2^{m+1},F} \otimes [\gamma,\delta)_{2^{m+1},F}$ is of exponent dividing $2^m$, the algebra $[\alpha_1,\beta)_{2,F} \otimes [\gamma_1,\delta)_{2,F}$ is split, which means $[\gamma_1,\delta)_{2,F}=[\alpha_1,\beta)_{2,F}$.
By the chain lemma for quaternion algebras (\cite[Section 14, Theorem 7]{Draxl:1983}), there exist $\epsilon \in F$ such that
$$[\gamma_1,\delta)_{2,F}=[\epsilon,\delta)_{2,F}=[\epsilon,\beta)_{2,F}=[\alpha_1,\beta)_{2,F}.$$
By Lemma \ref{commonslotchar2}, $[\alpha,\beta)_{2^{m+1},F} \otimes [(\epsilon,0,\dots,0),\beta)_{2^{m+1},F}^{\operatorname{op}}$ is of symbol length 1 in $\operatorname{Br}_{2^m}(F)$, the algebra $[(\epsilon,0,\dots,0),\delta)_{2^{m+1},F} \otimes [(\epsilon,0,\dots,0),\beta)_{2^{m+1},F}^{\operatorname{op}}$ is of symbol length at most 2 and $[(\epsilon,0,\dots,0),\delta)_{2^{m+1},F} \otimes [\gamma,\delta)_{2^{m+1},F}$ of symbol length 1.
Therefore, $[\alpha,\beta)_{2^{m+1},F} \otimes [\gamma,\delta)_{2^{m+1},F}$ is of symbol length at most 4 in $\operatorname{Br}_{2^m}(F)$.

Suppose now $\operatorname{char}(F)=3$. If $[\alpha,\beta)_{3^{m+1},F} \otimes [\gamma,\delta)_{3^{m+1},F}$ is of exponent dividing $3^m$, the algebra $[\alpha_1,\beta)_{3,F} \otimes [\gamma_1,\delta)_{3,F}$ is split, which means $[\gamma_1,\delta)_{3,F}=[-\alpha_1,\beta)_{3,F}$.
By the chain lemma for symbol algebras of degree 3 (see \cite[Corollary 7.2]{MatzriVishne:2014}), there exist $a,c \in F^\times$ and $b\in F$ such that either
$$[\gamma_1,\delta)_{3,F}=[\gamma_1,a)_{3,F}=[b,a)_{3,F}=[b,c)_{3,F}=[-\alpha_1,c)_{3,F}=[-\alpha_1,\beta)_{3,F}, \ \text{or}$$
$$[\gamma_1,\delta)_{3,F}=[\gamma_1,a)_{3,F}=[b,a)_{3,F}=[b,c)_{3,F}=[\alpha_1,c)_{3,F}=[\alpha_1,\beta^{-1})_{3,F}.$$
We continue with the first case, the proof of the second case is similar.
By Lemma \ref{commonslotchar2}, each of the following algebras $[\gamma,\delta)_{3^{m+1},F}\otimes [\gamma,a)_{3^{m+1},F}^{\operatorname{op}}$, $[(b,0,\dots,0),a)_{3^{m+1},F} \otimes [(b,0,\dots,0),c)_{3^{m+1},F}^{\operatorname{op}}$, and $[-\alpha,c)_{3^{m+1},F} \otimes [-\alpha,\beta)_{3^{m+1},F}^{\operatorname{op}}$ is of symbol length at most 3 in $\operatorname{Br}_{3^m}(F)$,
and each of the following algebras 
$[\gamma,a)_{3^{m+1},F}\otimes [(b,0,\dots,0),a)_{3^{m+1},F}^{\operatorname{op}}$ and $[(b,0,\dots,0),c)_{3^{m+1},F}\otimes [-\alpha,c)_{3^{m+1},F}^{\operatorname{op}}$ is of symbol length at most 1.
Hence, $[\alpha,\beta)_{3^{m+1},F} \otimes [\gamma,\delta)_{3^{m+1},F}$ is of symbol length at most $3\cdot 3+2=11$.
\end{proof}

Note that Corollary \ref{Small} is a considerable improvement to the previously known bound, i.e., the bound appearing in Corollary \ref{TignolBound}. For example, when $r=2$ and $\ell=m+1$, Corollary \ref{TignolBound} gives the upper bound $2^{2(m+1)-m}=2^{m+1}$, whereas Corollary \ref{Small} gives 4.

\begin{thm}
Suppose $F$ is a field of $\operatorname{char}(F)=p$.
Then for any symbol $A$ in $\H_{p^m}^3(F)$ of exponent $p^{m-1}$, its symbol length in $\H_{p^{m-1}}^3(F)$ is at most $\binom{m+1+p^2}{2}$.
\end{thm}

\begin{proof}
Write $A=\omega \otimes \beta \otimes \gamma$.
Since the exponent of $A$ is $p^{m-1}$, the symbol $\omega_1 \otimes \beta \otimes \gamma$ is trivial, which means, by Theorem \ref{GilleTheorem}, that $\gamma$ is in the image of the map $\operatorname{Nrd}: [\omega_1,\beta)_{p,F} \rightarrow F$.
The algebra $[\omega_1,\beta)_{p,F}$ is spanned as an $F$-vector space by the elements $\theta^iy^j$ for $i,j \in \{0,\dots,p-1\}$, where $\theta$ and $y$ satisfy $\theta^p-\theta=\omega_1$, $y^p=\beta$ and $y \theta y^{-1}=\theta+1$. 
Therefore there exist $x_{0,0},\dots,x_{p-1,p-1} \in F$ for which $\gamma=\operatorname{Nrd}(\sum_{i,j=0}^{p-1} x_{i,j} \theta^i y^j)$, hence $\gamma \in E=\mathbb{F}_p(\omega_1,\beta,x_{0,0},\dots,x_{p-1,p-1})$.
Consequently, $A$ descends to the symbol $A'=\omega \otimes \beta \otimes \gamma$ in $\H_{p^m}^3(E(\omega_2,\dots,\omega_m))$ whose exponent is $p^{m-1}$ because its $p^{m-1}$th power satisfies the norm condition for being split.
Therefore by \cite[Corollary 3.4]{ChapmanMcKinnie:2020}, the symbol length of $A'$ in $\H_{p^{m-1}}^3(E(\omega_2,\dots,\omega_m))$ is at most $\binom{m+1+p^2}{2}$, and therefore the symbol length of $A$ in $\H_{p^{m-1}}^3(F)$ is at most $\binom{m+1+p^2}{2}$.
\end{proof}

\begin{thm}
Suppose $F$ is a field of $\operatorname{char}(F)=2$.
Then for any symbol $A$ in $\H_{2^m}^{n+1}(F)$ of exponent $2^{m-1}$, its symbol length in $\H_{2^{m-1}}^{n+1}(F)$ is at most $\binom{m+n-1+2^n}{n}$.
\end{thm}

\begin{proof}
Write $A=\omega \otimes \beta_1 \otimes \dots \otimes \beta_n$.
Since the exponent of $A$ is $2^{m-1}$, the symbol $\omega_1 \otimes \beta_1 \otimes \dots \otimes \beta_n$ is trivial, which means, by Theorem \ref{QuadraticNorm}, that $\beta_n$ is represented by $\varphi$ where $\varphi=\langle \! \langle \beta_1,\dots,\beta_{n-1},\omega_1 ] \! ]$.
Therefore there exist $x_1,\dots,x_{2^n} \in F$ for which $\beta_n=\varphi(x_1,\dots,x_{2^n})$, hence $\beta_n \in E=\mathbb{F}_2(\omega_1,\beta_1,\dots,\beta_{n-1},x_1,\dots,x_{2^n})$.
Consequently, $A$ descends to the symbol $A'=\omega \otimes \beta_1 \otimes \dots \otimes \beta_n$ in $\H_{2^m}^n(E(\omega_2,\dots,\omega_m))$ whose exponent is $2^{m-1}$ because its $2^{m-1}$th power satisfies the norm condition for being split.
Because the $2$-rank of $E(\omega_2,\dots,\omega_m)$ is at most $m+n-1+2^n$, we can use \cite[Corollary 3.4]{ChapmanMcKinnie:2020} to bound the symbol length of $A'$ in $\H_{2^{m-1}}^3(E(\omega_2,\dots,\omega_m))$ from above by $\binom{m+n-1+2^n}{n}$, and therefore the symbol length of $A$ in $\H_{2^{m-1}}^3(F)$ is at most $\binom{m+n-1+2^n}{n}$.
\end{proof}

\section{Characteristic not $p$ analogue}

Here we study the analogous result to Theorem \ref{BrauerSymbol} in the case of fields of characteristic not $p$.
Recall that when $\operatorname{char}(F) \neq p$ and $F$ contains a primitive $p^m$th root of unity $\rho$, the group $\operatorname{Br}_{p^m}(F)$ is generated by $p^m$-symbol algebras 
$$(\alpha,\beta)_{p^m,F}=F\langle x,y :x^{p^m}=\alpha, y^{p^m}=\beta, yx=\rho xy \rangle,$$
for some $\alpha,\beta \in F^\times$.

\begin{thm}[{\cite{Tignol:1983}}]\label{pspecial}
Suppose $F$ is a field of characteristic not $p$ containing a primitive $p^t$th root of unity $\rho$.
Then if $(\alpha,\beta)_{p^t,F}$ is of exponent dividing $p^m$ with $m\leq t$, its symbol length in $\operatorname{Br}_{p^m}(F)$ is at most $p^{t-m}$.
\end{thm}

\begin{lem}\label{commonslot}
Given $\alpha,\beta,\gamma \in F^\times$, if $(\alpha,\beta)_{p,F}=(\gamma,\beta)_{p,F}$, then $(\alpha,\beta)_{p^{m+1},F}^{\operatorname{op}} \otimes (\gamma,\beta)_{p^{m+1},F}$ is of symbol length at most $p$ in $\H_{p^m}^2(F)$.
\end{lem}
\begin{proof}
If $(\alpha,\beta)_{p,F}=(\gamma,\beta)_{p,F}$ then $\gamma=\alpha \norm_{K/F}(f)$ for some $f\in K=F[\sqrt[p]{\beta}]$. Therefore, $(\gamma,\beta)_{p^{m+1},F}=(\alpha \norm_{K/F}(f),\beta)_{p^{m+1},F} \sim_{\operatorname{Br}} (\alpha,\beta)_{p^{m+1},F} \otimes (\norm_{K/F}(f),\beta)_{p^{m+1},F}$.
Therefore, 
$$(\gamma,\beta)_{p^{m+1},F} \otimes (\alpha,\beta)_{p^{m+1},F}^{\operatorname{op}} \sim_{\operatorname{Br}} (\norm_{K/F}(f),\beta)_{p^{m+1},F}.$$ Since $(\norm_{K/F}(f),\beta)_{p^{m+1},F}^{\otimes p^m} \sim_{\operatorname{Br}} (\norm_{K/F}(f),\beta)_{p,F} \sim_{\operatorname{Br}} F$, the class of  $(\norm_{K/F}(f),\beta)_{p^{m+1},F}$ is of exponent dividing $p^m$, and so its symbol length in $\operatorname{Br}_{p^m}(F)$ is bounded from above by $p$ by Theorem \ref{pspecial}.
\end{proof}

\begin{cor}
\
\begin{itemize}
\item[a.] If $\operatorname{char}(F) \neq 2$ and $F$ contains a primitive $2^{m+1}$th root of unity, then any class in $\operatorname{Br}_{2^m}(F)$ of symbol length at most 2 in $\operatorname{Br}_{2^{m+1}}(F)$ has symbol length at most 6 in $\operatorname{Br}_{2^m}(F)$.
\item[b.] If $\operatorname{char}(F) \neq 3$ and $F$ contains a primitive $3^{m+1}$th root of unity, then any class in $\operatorname{Br}_{3^m}(F)$ of symbol length at most 2 in $\operatorname{Br}_{3^{m+1}}(F)$ has symbol length at most $15$ in $\operatorname{Br}_{3^m}(F)$.
\end{itemize}
\end{cor}

\begin{proof}
Suppose $\operatorname{char}(F) \neq 2$ and that $F$ contains a primitive $2^{m+1}$th root of unity.
If $(\alpha,\beta)_{2^{m+1},F} \otimes (\gamma,\delta)_{2^{m+1},F}$ is of exponent dividing $2^m$, the algebra $(\alpha,\beta)_{2,F} \otimes (\gamma,\delta)_{2,F}$ is split, which means $(\gamma,\delta)_{2,F}=(\alpha,\beta)_{2,F}$.
By the chain lemma for quaternion algebras (\cite[Section 14, Theorem 7]{Draxl:1983}), there exist $\epsilon \in F^\times$ such that
$$(\gamma,\delta)_{2,F}=(\gamma,\epsilon)_{2,F}=(\alpha,\epsilon)_{2,F}=(\alpha,\beta)_{2,F}.$$
By Lemma \ref{commonslot}, $(\alpha,\beta)_{2^{m+1},F} \otimes (\alpha,\epsilon)_{2^{m+1},F}^{\operatorname{op}}$ is of symbol length at most 2 in $\operatorname{Br}_{2^m}(F)$, and so are $(\alpha,\epsilon)_{2^{m+1},F} \otimes (\gamma,\epsilon)_{2^{m+1},F}^{\operatorname{op}}$ and $(\gamma,\epsilon)_{2^{m+1},F} \otimes (\gamma,\delta)_{2^{m+1},F}$.
Therefore, $(\alpha,\beta)_{2^{m+1},F} \otimes (\gamma,\delta)_{2^{m+1},F}$ is of symbol length at most 6 in $\operatorname{Br}_{2^m}(F)$.

Suppose $\operatorname{char}(F)\neq 3$ and that $F$ contains a primitive $3^{m+1}$th root of unity. If $(\alpha,\beta)_{3^{m+1},F} \otimes (\gamma,\delta)_{3^{m+1},F}$ is of exponent dividing $3^m$, the algebra $(\alpha,\beta)_{3,F} \otimes (\gamma,\delta)_{3,F}$ is split, which means $(\gamma,\delta)_{3,F}=(\alpha^{-1},\beta)_{3,F}$.
By the chain lemma for symbol algebras of degree 3 (see \cite{Rost:1999}), there exist $a,b,c \in F^\times$ such that
$$(\gamma,\delta)_{3,F}=(\gamma,a)_{3,F}=(b,a)_{3,F}=(b,c)_{3,F}=(\alpha^{-1},c)_{3,F}=(\alpha^{-1},\beta)_{3,F}.$$
Now, each of the algebras $(\gamma,\delta)_{3^{m+1},F}\otimes (\gamma,a)_{3^{m+1},F}^{\operatorname{op}}$, $(\gamma,a)_{3^{m+1},F}\otimes (b,a)_{3^{m+1},F}^{\operatorname{op}}$, $(b,a)_{3^{m+1},F} \otimes (b,c)_{3^{m+1},F}^{\operatorname{op}}$,$(b,c)_{3^{m+1},F}\otimes (\alpha^{-1},c)_{3^{m+1},F}^{\operatorname{op}}$ and $(\alpha^{-1},c)_{3^{m+1},F} \otimes (\alpha^{-1},\beta)_{3^{m+1},F}^{\operatorname{op}}$ is of symbol length at most 3 in $\operatorname{Br}_{3^m}(F)$ , and so $(\alpha,\beta)_{3^{m+1},F} \otimes (\gamma,\delta)_{3^{m+1},F}$ is of symbol length at most $15$.
\end{proof}

\section{Where norm conditions fail}\label{section:norm}

We conclude the paper with an example that demonstrates how norm conditions for splitness do not extend to single symbols in $\H_{p^2}^{3}(F)$.
To be more precise, assuming $\omega \otimes \beta \in \H_{p^2}^2(F)$ is of exponent $p$, the symbol $\omega \otimes \beta \otimes \gamma \in \H_{p^2}^{3}(F)$ is split if $\gamma = a\cdot b^p$ where $a$ is a reduced norm in the cyclic algebra $[\omega,\beta)_{p^2,F}$ and $b \in F^\times$. The converse, however, is not necessarily true as we shall now see, i.e., we shall construct an example where $\gamma$ is not in $\operatorname{Nrd}([\omega,\beta)_{p^2,F}) \cdot (F^\times)^p$ and still $\omega \otimes \beta \otimes \gamma$ is trivial in $H_{p^2}^3(F)$.
We do this by following the analogous example in \cite{Merkurjev:1995}, which deals with the case of $\operatorname{char}(F) \neq p$.
We focus on the differences, and leave the statements and arguments that do not change with the characteristic.

Let $k$ be a field of $\operatorname{char}(k)=p$ for some odd prime $p$, and consider $E=k(x,y)$ the function field in two algebraically independent variables over $k$.
Let $K=E(s,t)$ be the function field in two algebraically independent variables over $E$.
Let $C=[x,s)_{p,K}$, $D=[y,t)_{p,K}$, and $T=C \otimes D$.
Write $L=E[\wp^{-1}(x),\wp^{-1}(y)]$.
Write $X$ for the Severi-Brauer variety of $[y,x)_{p,E}$.
For any field $E' \supseteq E$ for which $[y,x)_{p,E'}$ is a division algebra, we denote by $E'(X)$ the function field over $E'$ of $X$. Note that $T$ is represented by a single symbol $\omega \otimes \beta$ in $\H_{p^2}^2(K)$, and so $T\otimes x=\omega \otimes \beta \otimes x$ is a single symbol in $\H_{p^2}^3(K)$.

\begin{claim}
The algebra $T \otimes K(X)$ is a division algebra.
\end{claim}

\begin{proof}
%Since $F(X)$ is a purely inseparable extension of a purely transcendental extension of $F$ (explicitly $F(X)=F(x_1,\dots,x_p)[\sqrt[p]{\frac{1}{b}\varphi(x_1,\dots,x_p)}]$ where $\varphi$ is the norm form from $F[\wp^{-1}(c)]$ to $F$) it does not split any separable extension of $F$, and in particular it does not split the extension $L/F$, and therefore $L(X)=L \otimes_F F(X)$ is a field.
Since $E(X)/E$ is a regular extension, $L(X)=L \otimes_E E(X)$ is a field.
Now, $T \otimes K(X)$ is a generic crossed product with $L(X)\otimes_E K$ as the maximal Galois subfield, and therefore $T \otimes K(X)$ is a division algebra.
\end{proof}

Set $F=K(X)$.
For simplicity, we denote $T_F$ for $T \otimes F$.
This algebra represents a class in $\H_p^2(F)$.
Now, $T_{F} \otimes x$ is clearly trivial in $\H_{p^2}^3(F)$, 
so this is the example we want to focus on,
i.e., we want to explain why $b$ is not in $\operatorname{Nrd}(T_F) \cdot (F^\times)^p$, despite $T_F \otimes x$ being split.

\begin{claim}
The intersection of the group $\operatorname{Nrd}(T_{K(X)}) \cdot (K(X)^\times)^p$ with $E(X)^\times$ is $\operatorname{Norm}(L(X)/E(X)) \cdot (E(X)^\times)^p$.
\end{claim}

\begin{proof}
We start by noting that $K(X)=E(X)(s,t)$.
Consider the right-to-left $(s,t)$-adic valuation (see \cite{TignolWadsworth:2015} for reference).
The algebra $T_{K(X)}$ is defectless with value group $\frac{1}{p} \mathbb{Z} \times \frac{1}{p} \mathbb{Z}$ and residue field $L(X)$ (see \cite[Proposition 3.38]{TignolWadsworth:2015}).
If a given $\alpha \in F^\times$ is in $\operatorname{Nrd}(T_{K(X)}) \cdot (K(X)^\times)^p$, then the equation $\alpha=a \cdot b^p$ has a solution for $a \in \operatorname{Nrd}(T_{K(X)})$ and $b \in K(X)^\times$, but then there is also a solution over the residues $\alpha=\overline{a} \cdot \overline{b}^p$ where now $\overline{a} \in \operatorname{Norm}(L(X)/E(X))$ and $\overline{b} \in E(X)^\times$ (see \cite[Lemma 11.16]{TignolWadsworth:2015}).
\end{proof}

\begin{claim}
The intersection of the $\operatorname{Norm}(L(X)/E(X)) \cdot (E(X)^\times)^p$ with $E^\times$ is $\operatorname{Norm}(L/E) \cdot (E^\times)^p$.
\end{claim}

\begin{proof}
Follows the lines of \cite[Proposition 2.6]{Merkurjev:1995} without any change, for the original proof is characteristic free.
\end{proof}

As a conclusion, we can say that $x$ is not in $\operatorname{Nrd}(T_{K(X)}) \cdot (K(X)^\times)^p$, because it belongs to the intersection of the latter with $E^\times$, which is $\operatorname{Norm}(L/E) \cdot (E^\times)^p$.
Now, $x$ is not in $\operatorname{Norm}(L/E) \cdot (E^\times)^p$, because if it were, then $T \otimes x$ would be trivial in $\H_p^3(K)$, but it is not because the class $T \otimes x$ is the same as $[y,t)_{p,K} \otimes x=[y,x)_{p,K}^{\operatorname{op}} \otimes t$ and $[y,x)_{p,E}$ is a division algebra and $K=E(s,t)$ is a rational function field.

\section*{Acknowledgements}

The authors thank Eliyahu Matzri for bringing to our attention important works in the literature, which improved the quality of the paper.
The first author acknowledges the receipt of the Chateaubriand Fellowship (969845L) offered by the French Embassy in Israel in the fall of 2020, which helped establish the scientific connection with the second author. The third author was supported by a grant from the Simons Foundation (580782).

\section*{Statements}
Data Statement - Data sharing not applicable to this article as no datasets were generated or analysed during the current study.

Conflict of Interest Statement - On behalf of all authors, the corresponding author states that there is no conflict of interest.

\bibliographystyle{alpha}
\def\cprime{$'$}

\end{document}